\documentclass[12pt,oneside]{article}
\usepackage[all]{xy}
\usepackage{amsmath,amssymb,amsfonts,amsthm,mathrsfs,dsfont,mathtools}
\usepackage[english]{babel}
\usepackage{graphicx,stackrel}
\usepackage{multirow,rotating,paralist}
\usepackage[toc,page,header]{appendix}
\DeclareMathAlphabet{\mathpzc}{OT1}{pzc}{m}{it}
\usepackage{minitoc}
\usepackage{extarrows}
\usepackage{hyperref}
\usepackage{csquotes}

 \small {\par\vskip8pt minus3pt\rm}
\newcounter{item}[section]
\newcounter{kirshr}
\newcounter{kirsha}
\newcounter{kirshb}

\newtheorem{theorem}{Theorem}[section]

\newtheorem{lemma}[theorem]{Lemma}
\newtheorem{corollary}[theorem]{Corollary}
\newtheorem{proposition}[theorem]{Proposition}

\newtheorem{remark}[theorem]{Remark}

\newtheorem{definition}[theorem]{Definition}

\newcommand\undersym[2]{\raisebox{-6pt}{\tiny$#2$}{\kern-5pt}\mbox{$#1$}}
\newcommand\overcirc[1]{\raisebox{10pt}{\tiny$\circ$}{\kern-7pt}\mbox{$#1$}}

\title{\Large{Gap sequences of 1-Weierstrass points on non-hyperelliptic curves of genus $10$ }}
\date{}
\author{{\small Mohammed A. Saleem$^{1}$ and Eslam E. Badr$^{2}$}}

\begin{document}
\bibliographystyle{plain}
\maketitle                     
\vspace*{-.95cm}

\begin{center}
${^1}$Department of Mathematics, Faculty of Science,\\[-0.15cm]
Sohag University, Sohag, Egypt \\[-0.15cm]
${^2}$Department of Mathematics, Faculty of Science, \\[-0.15cm]
Cairo University, Giza, Egypt \\[-0.15cm]
\vspace*{.3cm}

\end{center}

\begin{center}
$
\begin{array}{ll}
\text{Emails:}&\text{abuelhassan@science.sohag.edu.eg} \\
  &\text{eslam@sci.cu.edu.eg}
\end{array}
$
\end{center}

\maketitle \bigskip
\begin{abstract}
In this paper, we compute the 1-gap sequences of 1-Weierstrass points of non-hyperelliptic smooth projective curves of genus $10$. Furthermore, the geometry of such points is classified as flexes, sextactic and tentactic points. Also, an upper bounds for their numbers are estimated.\\\\
\textbf{MSC 2010}: 14H55, 14Q05\\\\
\textbf{Keywords}: $1$-Weierstrass points, $q$-gap sequence, flexes , sextactic points, tentactic points, kanonical linear system, ,  Kuribayashi sextic curve.
\end{abstract}

\newpage

\section{Introduction}
\par Weierstrass points on curves have been extensively studied, in connection with many problems. For
example, the moduli space $M_g$ has been stratified with subvarieties whose points are isomorphism
classes of curves with particular Weierstrass points. For more deatails, we refer for example \cite{I1}, \cite{I2}.
\par At first, the theory of the Weierstrass points was developed only for smooth curves, and for their canonical divisors. In the last years, starting from some papers by \emph{R. Lax} and \emph{C. Widland} (see \cite{I3}, \cite{I4}, \cite{I5}, \cite{I6}, \cite{I7}, \cite{I8}), the theory has been reformulated for \emph{Gorenstein curves},
where the invertible dualizing sheaf substitutes the canonical sheaf. In this contest, the singular
points of a Gorenstein curve are always Weierstrass points.
\par In \cite{I9}, \emph{R. Notari } developed a technique to compute the Weierstrass gap sequence at a given point, no matter if it is simple or singular, on a plane curve, with respect to any linear system $V\subseteq H^0\big(C,O_C(n)\big)$. This technique can be useful to construct examples of curves with Weierstrass points of given weight, or to look for conditions for a sequence to be a Weierstrass gap Sequence. He used this technique to compute the Weierstrass gap sequence at a point of particular curves
and of families of quintic curves.
\par In this paper, we compute the $1$-gap sequence of the $1$-Weierstrass points on smooth non-hyperelliptic algebraic curves of genus $10$ which can be embedded holomorphically and effectively in algebraic curves of degree $6$. Furthermore, the geometry of Weierstrass points is classified as flexes, sextactic and tentactic points. On the other hand, we show that a smooth non-hyperelliptic curve of genus $10$ has no $4$-flex, $1,2$-sextactic or $9$-tentactic points. Also, an upper bound of the numbers of flexes, sextactic and tentactic points on such curves is estimated.

\section{Preliminaries}
\noindent\textbf{Notations.} We assume the following notations throughout the present paper.\\\\
$I(C_1, C_2; p)$; the intersection number of the curves $C_1$ and $C_2$ at the point $p$ \cite{pa1}.\\\\
$G_p^{(q)}(Q)$; the $q$-gap sequence of the point $p$ with respect to the linear system $Q$ \cite{pa3, pa4}.\\\\
$\omega^{(q)}(p)$; the $q$-weight of the point $p$ \cite{pa13}.\\\\
$N^{(q)}(C)$; the number of $q$-Weierstrass points on $C$ \cite{H7}.\\\\
$Q\big(-(\ell.p)\big);$ the set of divisors in the linear system $Q$ with multiplicity at least $\ell$ at the point $p$ \cite{pa13}.

\begin{lemma}\label{lemm1}
\em{\cite{pa7, pa13}} Let $X$ be a smooth projective plane curve of genus $g$. The number of $q$-Weierstrass points $N^{(q)}(C),$ counted with their $q$-weights, is given by
\[
N^{(q)}(C)=\left\{
\begin{array}
[c]{lr}%
g(g^{2}-1),\,\,\,\,\,\,\,\,\,\,\,\,\,\,\,\,\,\,\,\,\,\,\,\,\,\,\,\,\,\,\,\,\,\,\,\,\,\,if\, q=1 & \\\\
(2q-1)^{2}(g-1)^{2}g,\,\,\,\,\,\,\,\,\,\,\,\,\,\,\,if\, q\geq2. &
\end{array}
\right.
\]
In particular, for smooth projective plane sextic $($i.e.\,\,$g=10)$, the number
of $1$-Weierstrass points is $990$ counted with their weights.
\end{lemma}

\begin{theorem}\label{thm1}
\em{\cite{pa13}} Let $X$ be a non-hyperelliptic curve of genus $\geq3.$ Write $G^{(1)}_p(Q)=\{n_1<n_2<...<n_g\}$, then
\begin{description}
    \item[(1)] $n_1=1,$
    \item[(2)] $n_r\leq2r-2$ for every $r\geq2,$
    \item[(3)] $\omega^{(1)}(p)\leq \dfrac{(g-1)(g-2)}{2}.$
    \item[(4)] There are at least $2g+6$ 1-Weierstrass points on $X.$
\end{description}
\end{theorem}

\begin{remark}
\em{For more details on $q$-Weierstrass points on Riemann surfaces, we refer for example to \cite{pa13, pa6}.}
\end{remark}


\section{Main results}
\par Let $X$ be a smooth projective plane curve of genus $10$, and let $Q:=|K|$ be its canonical linear system.
\begin{proposition}
The linear system $Q$ is $g^{9}_{18}.$
\end{proposition}

\begin{proof}
The result is an immediate consequence, since \[dim\,Q:=dim\,|K|=g-1=9,\]
and
\[deg\,(Q):=deg\,(K)=2(g-1)=18.\]

\end{proof}

\begin{corollary}
Let $p\in X$, then $\sharp G^{(1)}_p(Q)=10$ and $G^{(1)}_p(Q)\subset\{1,2,3,...,19\}.$
\end{corollary}

\begin{lemma}
The set of cubic divisors on $X$ form a linear system which is $g^{9}_{18}$.
\end{lemma}

\subsection{Flexes}
\begin{definition}
\em{\cite{pa3, pa14}} A point $p$ on a smooth plane curve $C$ is said to
be a flex point if the tangent line $L_p$ meets $C$ at $p$ with
contact order $I_{p}(C,L_{p};p)$ at least three. We say that $p$ is
$i$-flex, if $I_{p}(C,L_{p})-2=i$. The positive integer $i$ is
called the flex order of $p.$
\end{definition}
\noindent Our main results for this part are the following.

\begin{theorem}\label{thm2}
Let $p$ be a flex point on a smooth projective non-hyperelliptic plane curve $C$ of $g=10$. Let $L_p$ be the tangent lint to $C$ at $p$ such that $I(C, E_p)=\mu_f.$ Then $G^{(1)}_p(Q)=\{1,2,3,1+\mu_f,2+\mu_f,3+\mu_f,2\mu_f+1,2\mu_f+2,3\mu_f+1,3\mu_f+2\}.$ Moreover, the geometry of such points is given by the following table:
\[
\begin{tabular}{|c|c|c|}
  \hline
  $\omega^{(1)}_p(Q)$ & $G_p(Q)$ & Geometry \\\hline
  2 & \{1,2,3,...,8,10,11\} & \emph{1-flex} \\\hline
  15 & \{1,2,3,5,6,7,9,10,13,14\} & \emph{2-flex} \\\hline
  28  & \{1,2,3,6,7,8,11,12,16,17\} & \emph{3-flex} \\\hline
  \end{tabular}
\]
\end{theorem}

\begin{proof}
The dimensions of $Q(-1.p)$ and $Q(-2.p)$ do not depend on whether $p$ is 1-Weierstrass point or not, i.e., \[1,2\in G^{(1)}_p(Q).\] The spaces $Q(-3.p)=...=Q(-\mu_f.p),$ consists of divisor of cubic curves of the form $L_pR$, where $R$ is an arbitrary conic. Hence, $dim\,Q(-\ell.p)=6$ for $\ell=3,...,\mu_f.$ That is, \[3\in G^{(1)}_p(Q).\]
The space $Q\big(-(1+\mu_f).p\big)$ consists of divisor of cubic curves of the form $L_pR$, where $R$ is a conic passing through $p.$ Hence, $dim\,Q\big(-(1+\mu_f).p\big)=5.$ That is, \[1+\mu_f\in G^{(1)}_p(Q).\]
The space $Q\big(-(2+\mu_f).p\big)$ consists of divisor of cubic curves of the form $L_pR$, where $R$ is a conic passing through $p$ with contact order at least $2.$ Hence, $dim\,Q\big(-(2+\mu_f).p\big)=4.$ That is, \[2+\mu_f\in G^{(1)}_p(Q).\]
The spaces $Q\big(-(3+\mu_f).p\big)=...=Q\big(-2\mu_f.p\big)$ consists of divisor of cubic curves of the form $L^2_pH$, where $H$ is an arbitrary hyperplane. Hence, $dim\,Q\big(-\ell.p\big)=3$ for $\ell=3+\mu_f,...,2\mu_f.$ That is, \[3+\mu_f\in G^{(1)}_p(Q).\]
The space $Q\big(-(2\mu_f+1).p\big)$ consists of divisor of cubic curves of the form $L^2_pH$, where $H$ is a hyperplane through $p$. Hence, $dim\,Q\big(-(2\mu_f+1).p\big)=2$. That is, \[2\mu_f+1\in G^{(1)}_p(Q).\]
The spaces $Q\big(-(2\mu_f+2).p\big)=...=Q\big(-3\mu_f.p\big)$ consists of divisor of cubic curves of the form $L^2_pH$, where $H$ is a hyperplane through $p$ with contact order at least $2$. Hence, $dim\,Q\big(-(2\mu_f+2).p\big)=1.$ That is, \[2\mu_f+2\in G^{(1)}_p(Q).\]
The space $Q\big(-(3\mu_f+1).p\big)$ consists of divisor of  the cubed tangent line $L^3_p.$ Hence, $dim\,Q\big(-(3\mu_f+1).p\big)=0.$ That is, \[3\mu_f+1\in G^{(1)}_p(Q).\]
The spaces $Q\big(-\ell.p\big)=\phi$ for $\ell\geq3\mu_f+2.$ Consequently, \[3\mu_f+2\in G^{(1)}_p(Q).\]
Consequently, \[G^{(1)}_p(Q)=\{1,2,3,1+\mu_f,2+\mu_f,3+\mu_f,2\mu_f+1,2\mu_f+2,3\mu_f+1,3\mu_f+2\}.\] Finally, it is well known that curves of genus $10$ can be embedded into algebraic curves of degree $6$. Hence, by the famous Bezout\textquoteright s theorem, the tangent line meets $C$ at the flex point $p$ with $3\leq\mu_f\leq6.$ On the other hand, by \emph{Theorem \ref{thm1}}, it follows that, $\mu_f\neq6.$\\
Which completes the proof.
\end{proof}

\begin{corollary}
If a smooth projective curve $X$ of $g=10$ has $4$-flex points, then $X$ is hyperelliptic.
\end{corollary}

\begin{corollary}
On a smooth non-hyperelliptics projective plane curve $C$ of $g=10$, the $1$-weight of a flex point is given by
\[\omega^{(1)}_p(Q)=13\mu_f-37,\]
where $\mu_f$ is the multiplicity of the tangent line $L_p$ to $C$ at $p.$
\end{corollary}

\begin{proof}
Let  $p$ be a flex point on $C$, then by \emph{Theorem \ref{thm2}} \[G^{(1)}_p(Q)=\{1,2,3,1+\mu_f,2+\mu_f,3+\mu_f,2\mu_f+1,2\mu_f+2,3\mu_f+1,3\mu_f+2\}.\]
Consequently,
\begin{eqnarray*}
\omega^{(1)}_p(Q)&:=&\Sigma_{r=0}^{g}(n_r-r)\\
&=&(1+\mu_f-4)+(2+\mu_f-5)+(3+\mu_f-6)+(2\mu_f+1-7)+(2\mu_f+2-8)\\
&+&(3\mu_f+1-9)+(3\mu_f+2-10)\\
&=&13\mu_f-37.
\end{eqnarray*}
\end{proof}

\noindent\textbf{Notation.} Let $F_i^{(q)}(C)$ be the set of $i$-flex points which are $q$-Weierstrass points on $C.$ Also, let $NF_{i}^{(q)}(C)$ denotes the cardinality of $F_i^{(q)}(C)$.
\begin{corollary}
For a smooth non-hyperelliptic projective plane curve $C$ of $g=10$, the maximal cardinality of $F_i^{(1)}$ is given by the following table:
\[
\begin{tabular}{|c|c|}
  \hline
  $i$ & Maximum $NF_{i}^{(1)}(C)$ \\\hline
  1 & $\mathbf{495}$ \\\hline
  2 & $\mathbf{66}$ \\\hline
  3 & $\mathbf{35}$ \\\hline
  4 & $\mathbf{0}$ \\\hline
  \end{tabular}
\]
\end{corollary}

\begin{proof}
The number of the 1-Weierstrass points on $C$ is $990$ counted with their $1$-weight. Hence,
\[NF_{i}^{(1)}\leq [\frac{990}{13(i+2)-37}],\]
where $i=1,2,3.$
\end{proof}

\subsection{Sextactic Points}
\par In analogy with tangent lines and flexes of projective plane curves$,$ one can consider \it{osculating conics} and \textit{sextactic points} in the following way:
\begin{lemma}\em{\cite{pa4}}
Let $p$ be a non-flex point on a smooth projective plane curve $X$ of
degree $d\geq3.$ Then there is an unique irreducible conic $D_{p}$ with $%
I_{p}(X,D_{p};p)\geq 5$. This unique irreducible conic $D_{p}$ is called
the \textit{osculating conic of }$\mathit{X}$\textit{\ at }$\mathit{p}$%
\textit{.}
\end{lemma}

\begin{definition}\em{\cite{pa3}}
A non-flex point $p$ on a smooth projective plane curve $X$ is said to be a $sextactic$ $point$ if the osculating conic $D_{p}$ meets $X$ at $p$ with contact order at least six. A sextactic point $p$ is said to be $i$\textit{-sextactic}, if $%
I_{p}(X,D_{p};p)-5=i.$ The positive integer $i$ is called the sextactic order.
\end{definition}

\em{\noindent Now, the main results for this part are the following.}
\begin{theorem}\label{thm3}
Let $p$ be a sextactic point on a smooth projective non-hyperelliptic curve $C$ of $g=10$. Let $D_p$ be the osculating conic to $C$ at $p$ such that $I(C, D_p;p)=\mu_s.$ Then $G^{(1)}_p(Q)=\{1,2,3,...,7,1+\mu_s,2+\mu_s,3+\mu_s\}.$
\end{theorem}

\begin{proof}
Again, the dimensions of $Q(-1.p)$ and $Q(-2.p)$ do not depend on whether $p$ is 1-Weierstrass point or not, i.e., \[1,2\in G^{(1)}_p(Q).\]
Moreover, let $L_p$ be the tangent line to $C$ at $p$, then \[div\big(L_pR\big)\in Q(-2.p)-Q(-3.p),\]
where, $R_0$ is a conic not through $p.$ That is, \[3\in G^{(1)}_p(Q).\]
Furthermore, \[div\big(L_pR_1\big)\in Q(-3.p)-Q(-4.p),\]
where, $R_1$ is a conic passing through $p$ with multiplicity $1.$ That is, \[4\in G^{(1)}_p(Q).\]
Also,  \[div\big(L_p^2H_0\big)\in Q(-4.p)-Q(-5.p),\]
where, $H_0$ is a hyperplane not through $p.$ That is, \[5\in G^{(1)}_p(Q).\]
Similarly, \[div\big(L_p^2H_1\big)\in Q(-5.p)-Q(-6.p),\]
where, $H_1$ is a hyperplane passing through $p$ with multiplicity $1.$ So that, \[6\in G^{(1)}_p(Q).\]
Now, the spaces $Q(-7.p)=...=Q(-\mu_s.p)$ consists of divisors of cubic curves of the form $D_PH$, where $H$ is an arbitrary line. Hence, \[7\in G^{(1)}_p(Q).\]
On the other hand, the space $Q\big(-(1+\mu_s).p\big)$ consists of divisors of cubic curves of the form $D_PH$, where $H$ is a hyperplane through $p.$ Consequently, \[1+\mu_s\in G^{(1)}_p(Q).\]
Also, the space $Q\big(-(2+\mu_s).p\big)$ contains only the cubic divisor $D_PL_p$. Then, \[2+\mu_s\in G^{(1)}_p(Q).\]
Finally, $Q\big(-\ell.p\big)=\phi,$ for $\ell\geq3+\mu_s.$ That is, \[3+\mu_s\in G^{(1)}_p(Q).\]\\
Which completes the proof.
\end{proof}

\begin{corollary}
Let $p$ be a sextactic point on a smooth projective non-hyperelliptic curve $C$ of $g=10$. Then, the geometry of such points is given by the following table:
\[
\begin{tabular}{|c|c|c|}
  \hline
  $\omega^{(1)}_p(Q)$ & $G_p(Q)$ & Geometry \\\hline
  3 & \{1,2,3,...,,7,9,10,11\} & \emph{3-sextactic} \\\hline
  6 & \{1,2,3,...,,7,10,11,12\} & \emph{4-sextactic} \\\hline
  9  & \{1,2,3,...,,7,11,12,13\} & \emph{5-sextactic} \\\hline
  12  & \{1,2,3,...,,7,12,13,14\} & \emph{6-sextactic} \\\hline
  15  & \{1,2,3,...,,7,13,14,15\} & \emph{7-sextactic} \\\hline
  \end{tabular}
\]
\end{corollary}

\begin{proof}
It follows by \emph{Theorem \ref{thm3}} and Bezout\textquoteright s theorem, that $D_p$ meets $C$ at $p$ with $8\leq\mu_s\leq12.$ Hence, varying $\mu_s$ produces the last table.
\end{proof}

\begin{corollary}
If a smooth projective curve $X$ of $g=10$ has $1$-sextactic or $2$-sextactic points, then $X$ is hyperelliptic.
\end{corollary}

\noindent\textbf{Notation.} Let $S_i^{(q)}(C)$ be the set of $i$-sextactic points which are $q$-Weierstrass points on $C.$ Also, let $NS_{i}^{(q)}(C)$ denotes the cardinality of the set $S_i^{(q)}(C)$.
\begin{corollary}
For a smooth non-hyperelliptic projective curve $C$ of $g=10$, the maximal cardinality of $S_i^{(1)}$ is given by the following table:
\[
\begin{tabular}{|c|c|}
  \hline
  $i$ & Maximum $NS_{i}^{(1)}(C)$ \\\hline
  1 & $\mathbf{0}$ \\\hline
  2 & $\mathbf{0}$ \\\hline
  3 & $\mathbf{330}$ \\\hline
  4 & $\mathbf{165}$ \\\hline
  5 & $\mathbf{110}$ \\\hline
  6 & $\mathbf{82}$ \\\hline
  7 & $\mathbf{66}$ \\\hline
  \end{tabular}
\]
\end{corollary}

\begin{proof}
The number of the 1-Weierstrass points of such curves is $990.$ Hence,
\[NS_{i}^{(1)}\leq [\frac{990}{3(i+5)-21}],\]
where $i=3,4,5,6,7.$
\end{proof}

\begin{corollary}
On a smooth non-hyperelliptic projective plane curve $C$ of $g=10$, the $1$-weight of a sextactic point is given by
\[\omega^{(1)}_p(Q)=3\mu_s-21,\]
where $\mu_s$ is the multiplicity of the osculating conic $D_p$ at $p.$
\end{corollary}

\begin{proof}
If $p$ is a sextactic point on $C$, then \[G^{(1)}_p(Q)=\{1,2,3,...,7,1+\mu_s,2+\mu_s,3+\mu_s\}.\]
Consequently,
\begin{eqnarray*}
\omega^{(1)}_p(Q)&:=&\Sigma_{r=0}^{g}(n_r-r)\\
&=&(1+\mu_s-8)+(2+\mu_s-9)+(3+\mu_s-10)\\
&=&3\mu_s-21.
\end{eqnarray*}
\end{proof}


\subsection{Tentactic points}
\begin{definition}
\em{\cite{pa3, H6}}
A point $p$ on a smooth plane curve $C$ of genus $g\geq2$, which is neither flex nor sextactic point, is said to be a \emph{tentactic point}, if there a cubic $E_p$ which meets $C$ at $p$ with contact order at least $10.$ The positive integer $t$ such that $i:=I(C,E_p;p)-9$ is called the tentactic order of $p$. Moreover, the point $p$ is said to be \emph{i-tentactic}.
\end{definition}

\begin{theorem}
Let $p$ be a tentactic point on a smooth projective non-hyperelliptic curve $C$ of $g=10$ and let $E_p$ be its osculating cubic curve such that $I(C, E_p;p)=\mu_t.$ Then $G^{(1)}_p(Q)=\{1,2,3,...,9,1+\mu_t\}.$ Moreover, the geometry of such points is given by the following table:
\[
\begin{tabular}{|c|c|c|}
  \hline
  $\omega^{(1)}_p(Q)$ & $G_p(Q)$ & Geometry \\\hline
  1 & \{1,2,3,...,9,11\} & \emph{1-tentactic} \\\hline
  2 & \{1,2,3,...,9,12\} & \emph{2-tentactic} \\\hline
  3  & \{1,2,3,...,9,13\} & \emph{3-tentactic} \\\hline
  4  & \{1,2,3,...,9,14\} & \emph{4-tentactic} \\\hline
  5  & \{1,2,3,...,9,15\} & \emph{5-tentactic} \\\hline
  6  & \{1,2,3,...,9,16\} & \emph{6-tentactic} \\\hline
  7  & \{1,2,3,...,9,17\} & \emph{7-tentactic} \\\hline
  8  & \{1,2,3,...,9,18\} & \emph{8-tentactic} \\\hline
\end{tabular}
\]
\end{theorem}

\begin{proof}
Since, the point $p$ is neither flex nor sextactic, then \[dim\,Q(-\ell.p)=9-\ell\,\,\,for\,\,\ell=1,2,3,...,9.\]
Hence,\[1,2,3,...,9\in G^{(1)}_p(Q).\] Moreover, assuming that $I(C, E_p)=\mu_t,$ then \[div\,(E_p)\in Q(-\mu_t.p)-Q(-(1+\mu_t).p).\] Therefore, $1+\mu_t\in G^{(1)}_p(Q).$ Consequently, \[G_p(Q)=\{1,2,3,...,9,1+\mu_t\}.\] Finally, $19\notin G^{(1)}_p(Q)$ as $n_{10}\leq 18.$
\end{proof}

\begin{corollary}
If a smooth projective curve $X$ of $g=10$ has $9$-tentactic points, then $X$ is hyperelliptic.
\end{corollary}

\noindent\textbf{Notation.} Let $T_i^{(q)}(C)$ be the set of $i$-tentactic points which are $q$-Weierstrass points on $C.$ Also, let $NT_{i}^{(q)}(C)$ denotes the cardinality of the set $T_i^{(q)}(C)$.
\begin{corollary}
For a smooth projective non-hyperelliptic curve $C$ of $g=10$, the maximal cardinality of $T_i^{(1)}$ is given by the following table:
\[
\begin{tabular}{|c|c|}
  \hline
  $i$ & Maximum $NT_{i}^{(1)}(C)$ \\\hline
  1 & $\mathbf{990}$ \\\hline
  2 & $\mathbf{495}$ \\\hline
  3 & $\mathbf{330}$ \\\hline
  4 & $\mathbf{247}$ \\\hline
  5 & $\mathbf{198}$ \\\hline
  6 & $\mathbf{165}$ \\\hline
  7 & $\mathbf{141}$ \\\hline
  8 & $\mathbf{123}$ \\\hline
  9 & $\mathbf{0}$ \\
  \hline
\end{tabular}
\]
\end{corollary}

\begin{proof}
The number of the 1-Weierstrass points of a smooth projective non-hyperelliptic curve of genus $g$ is $(g-1)g(g+1).$ Hence, in particular,
\[NT_{i}^{(1)}\leq [\frac{990}{i}],\]
where $i=1,2,3,...,8.$
\end{proof}

\bigskip
{\fontsize{13.5}{9}\textbf{Concluding remarks.}}\\
We conclude the paper by the following remarks and comments.
\begin{itemize}
\item The 1-Weierstrass points of an algebraic curve of $g=10$ are classified as flexes, sextactic and tentactic points. In particular, the present authors computed the $1$-gap sequences of such points on non-hyperelliptic curves of genus $10$. Furthermore the geometry of these points is investigated and an upper bound for their number is estimated.
 \item The main theorems constitute a motivation to solve other problems. One of these problems is the investigation of the geometry of the $1$-Weierstrass points of Kuribayashi sextic curve with three parameters $a,b$ and $c$ defined by the equation:
     \[KI_{a,b,c}: X^6+Y^6+Z^6+aX^3Y^3+bX^3Z^3+cY^3Z^3.\]
     However, this problem will be the object of a forthcoming work.

\end{itemize}


\bigskip\noindent

\end{document}